\documentclass[12pt]{article}
\usepackage[pdftex]{color, graphicx}
\usepackage{setspace}
\usepackage[utf8]{inputenc}
\usepackage{amsmath}
\usepackage[pdftex,bookmarks,colorlinks]{hyperref}
\usepackage{latexsym}
\usepackage{amsfonts}
\usepackage{amssymb}
\usepackage{amsthm}
\usepackage{paralist}
\usepackage{mathrsfs}

\begin{document}\newtheorem{thm}{Theorem}
\newtheorem{cor}[thm]{Corollary}
\newtheorem{lem}{Lemma}
\theoremstyle{remark}\newtheorem{rem}{Remark}
\theoremstyle{definition}\newtheorem{defn}{Definition}
\title{Inequalities for Variation Operator}
\author{Sakin Demir}
\maketitle
\begin{abstract}Let $f$ be a measurable function defined on $\mathbb{R}$. For each $n\in\mathbb{Z}$ define the operator $A_n$ by
$$A_nf(x)=\frac{1}{2^n}\int_x^{x+2^n}f(y)\, dy.$$
Consider the variation operator
$$\mathcal{V}f(x)=\left(\sum_{n=-\infty}^\infty|A_nf(x)-A_{n-1}f(x)|^s\right)^{1/s}$$
for $2\leq s<\infty$.\\
It has been proved in \cite{jkw1} that  $\mathcal{V}$ is of strong type $(p,p)$ for $1<p<\infty$ and is of weak type $(1,1)$, it maps $L^\infty$ to BMO. We first provide a completely different proofs for these known results and in addition we prove that  $\mathcal{V}$ maps $H^1$ to $L^1$.  Furthermore,  we prove that it satisfies vector-valued weighted strong type and weak type inequalities. As a special case it follows that $\mathcal{V}$ satisfies weighted strong type and weak type inequalities.
\end{abstract}

\noindent
{\bf{Introduction:}}\\
\noindent
Let $f$ be a measurable function defined on $\mathbb{R}$. For each $n\in\mathbb{Z}$ define the operator $A_n$ by
$$A_nf(x)=\frac{1}{2^n}\int_x^{x+2^n}f(y)\, dy.$$
It is a well known problem to study the different kinds of convergence of the sequence $\{A_nf\}_n$ when $f\in L^p(\mathbb{R})$ for some $1\leq p<\infty$. \\
Consider the variation operator
$$\mathcal{V}f(x)=\left(\sum_{n=-\infty}^\infty|A_nf(x)-A_{n-1}f(x)|^s\right)^{1/s}$$
for $2\leq s<\infty$.\\
Analyzing the boundlessness of the variation operator $\mathcal{V}f$ is a method of measuring the speed of convergence of the sequence $\{A_nf\}$.\\

If a postive function $w\in L_{\text{loc}}^1(\mathbb{R})$ satisfies the following condition we say that $w$ is an $A_p$ weight for some $1<p<\infty$:\\
$$\sup_I\left(\frac{1}{|I|}\int_Iw\, dx\right)\left(\frac{1}{|I|}\int_Iw^{-\frac{1}{p-1}}\, dx\right)^{p-1}<\infty ,$$
where the supremum is taken over all intervals $I$ in $\mathbb{R}$.\\
$w$ is an $A_\infty$ weight if there exists $\delta >0$ and $\epsilon >0$ such that for any measurable $E\subset Q$,
$$|E|<\delta\cdot |Q|\implies w(E)<(1-\epsilon )\cdot w(Q).$$
Here
$$w(E)=\int_Ew$$
It is well known and easy to see that $w\in A_p\implies w\in A_\infty$ if $1<p<\infty$.\\
We say that $w\in A_1$ if there is a positive constant $C$ such that
$$\frac{1}{|Q|}\int_Qw(x)\, dx\leq Cw(x)$$
for a.e. $x\in Q$.\\

We say that an operator $T:L^p(X)\to L^p(X)$ maps $L^p(w)$ to itself if there is a postive constant $C$ such that
$$\int_{\mathbb{R}}|Tf(x)|^pw(x)\, dx\leq C\int_{\mathbb{R}}|f(x)|^pw(x)\, dx$$
for all $f\in L^p$.\\
We say that  $T$ is of strong type $(p,p)$ if there exists a positive constanr $C$ such that
$$\|Tf\|_p\leq C\|f\|_p$$
for all $f\in L^p(X)$. $T$ is of weak type $(1,1)$ (or satiesfies a weak type $(1,1)$ inequality) if there exists a positive constant $C$ such that
$$\mu \{x:|Tf(x)|>\lambda\}\leq \frac{C}{\lambda}\|f\|_1$$
for all $f\in L^1(X)$.\\
 We say that the kernel $K$ satisfies $D_r$ condition for $1\leq r<\infty$, and write $K\in D_r$, if there exists a sequence $\{c_l\}_{l=1}^\infty$ of positive numbers $\sum_lc_l<\infty$ and any $l\geq 2$ and $x>0$,
$$\left(\int_{S_l(x)}\|K(x-y)-K(-y)\|_{l^s}^r\, dy\right)^{1/r}\leq c_l|S_l(x)|^{-1/{r^\prime}},$$
where $S_l(x)=(2^lx, 2^{l+1}x)$.\\

When $K\in D_1$ we have the Hörmander condition:
$$\int_{\{y>4x\}}\|K(x-y)-K(-y)\|_{l^s}\, dy\leq C$$
where $C$ is a positive constant which does not depend on $y\in\mathbb{R}$.
Let $A$ and $B$ be Banach spaces. A linear operator $T$ mapping $A$-valued functions into $B$-valued functions is called a singular integral operator of convolution type if the following conditions are satisfied:
\begin{enumerate}[(i)]
\item $T$ is a bounded operator from $L_A^q(\mathbb{R})$ to $L_B^q(\mathbb{R})$ for some $q$, $1\leq q\leq\infty.$
\item There exists a kernel $K\in D_1$ such that 
$$Tf(x)=\int K(x-y)\cdot f(y)\,dy$$ 
for every $f\in L_A^q(\mathbb{R})$ with compact support and for a.e. $x\notin{\rm{supp}}(f)$.
\end{enumerate}

Given a locally integrable function $f$ we define the sequence-valued operator $T$ as follows:
\begin{align*}
Tf(x)&=\{A_nf(x)-A_{n-1}f(x)\}_n\\
&=\left\{\int_{\mathbb{R}}\left(\frac{1}{2^n}\chi_{(-2^n,0)}(x-y)-\frac{1}{2^{n-1}}\chi_{(-2^{n-1},0)}(x-y)\right)f(y)\, dy\right\}_n\\
&=\int_{\mathbb{R}}K(x-y)\cdot f(y)\, dy,
\end{align*}
where $K$ is the sequence-valued function
$$K(x)=\left\{\frac{1}{2^n}\chi_{(-2^n,0)}(x)-\frac{1}{2^{n-1}}\chi_{(-2^{n-1},0)}(x)\right\}_n.$$
It is clear that 
$$\|Tf(x)\|_{l^s}=\mathcal{V}f(x).$$
\begin{lem}\label{ksdr}The kernel operator
$$K(x)=\left\{\frac{1}{2^n}\chi_{(-2^n,0)}(x)-\frac{1}{2^{n-1}}\chi_{(-2^{n-1},0)}(x)\right\}_n$$
satisfies $D_r$ condition for $r\geq 1$.
\end{lem}
\begin{proof} Let $x_0\in\mathbb{R}$ and $i\in\mathbb{Z}$ be given, consider $x$ and $y$ in $\mathbb{R}$ such that $x_0<x\leq x_0+2^i$ and $x_0+2^j<y\leq x_0+2^{j+1}$ with $j>i$. Let $\phi_n(y)=\chi_{(-2^n,0)}(y)$. Then $\phi_n(x-y)-\phi_n(x_0-y)=0$ unless $n=j$ in which case 
$$\phi_j(x-y)-\phi_j(x_0-y)=\chi_{(x_0+2^j, x+2^j)}(y).$$
To see this first it is clear that $\phi_n(x-y)=\chi_{(x, x+2^n)}(y)$. Now if $n<i$ then
 $$x+2^n<x-x_0+x_0+2^i\leq x_0+2^j<y.$$
 Thus $\phi_n(x-y)=0$. Clearly, the same holds for $\phi_n(x_0-y)$. If $i\leq n<j$ then 
 $$x+2^n\leq x_0+2^i+2^n\leq x_0+2\cdot 2^n\leq x_0+2^j,$$
 and $\phi_n(x-y)=\phi_n(x_0-y)=0$. If $n>j$ then 
 $$x+2^n>x_0+2^n\geq x_0+2^{j+1}\geq y,$$
 and since $y>x>x_0$ we have
 $$\phi_n(x-y)-\phi_n(x_0-y)=1-1=0.$$
 Finally, if $n=j$ then
 $$\phi_j(x_0-y)=\chi_{(x_0,x_0+2^j)}(y)=0,$$
 while
 $$\phi_j(x-y)=\chi_{(x,x+2^j)}(y)=1$$
 whenever
 $$x_0+2^j\leq y\leq x+2^j.$$
We now have
\begin{align*}
\|K(x-y)-K(x_0-y)\|_{l^s}^s&=\sum_n\left|\frac{1}{2^n}\phi_n(x-y)-\frac{1}{2^{n-1}}\phi_{n-1}(x-y)\right.\\
&\quad -\left.\left(\frac{1}{2^n}\phi_n(x_0-y)-\frac{1}{2^{n-1}}\phi_{n-1}(x_0-y)\right)\right|^s\\
&=\sum_n\left|\frac{1}{2^n}\phi_n(x-y)-\frac{1}{2^n}\phi_n(x_0-y)\right.\\
&\quad -\left.\left(\frac{1}{2^{n-1}}\phi_{n-1}(x-y)-\frac{1}{2^{n-1}}\phi_{n-1}(x_0-y)\right)\right|^s\\
&=2\left|\frac{1}{2^j}\phi_j(x-y)-\frac{1}{2^j}\phi_j(x_0-y)\right|^s\\
&=2\left|\frac{1}{2^j}\chi_{(x_0+2^j, x+2^j)}(y)\right|^s.
\end{align*}
Thus we get
$$\|K(x-y)-K(x_0-y)\|_{l^s}=2^{1/s}\left|\frac{1}{2^j}\chi_{(x_0+2^j, x+2^j)}(y)\right|.$$
Given $x$, choose an integer $i$ such that $2^{i-1}\leq x <2^i$. By using our previous observation we obtain
\begin{align*}
\left(\int_{2^lx}^{2^{l+1}x}\|K(x-y)-K(-y)\|_{l^s}^r\, dy\right)^{1/r}\qquad\qquad\qquad\qquad\qquad\qquad\\
\leq \left(\int_{2^{l+i-1}}^{2^{l+i}}\|K(x-y)-K(-y)\|_{l^s}^r\right)^{1/r}\\
+ \left(\int_{2^{l+i}}^{2^{l+i+1}}\|K(x-y)-K(-y)\|_{l^s}^r\right)^{1/r}\quad\quad\\
\leq 2^{2/s}\frac{2^{i/r}}{2^{l+i}}=C2^{-l/r}|S_l(x)|^{-1/{r^\prime}}\qquad\quad
\end{align*}
and this completes our proof.
\end{proof}
It is easy to check that the Fourier transform of the kernel of our vector-valued operator $T$ is bounded and since we also have $D_1$ condition we deduce that $T$ is a singular integral operator of convolution type.
\begin{lem}\label{singmaps} A singular integral operator $T$ mapping $A$-valued functions into $B$-valued functions can be extended to an operator  defined in all $L_A^p$, $1\leq p<\infty$, and satisfying
\begin{enumerate}[(a)]
\item \label{sininta} $\|Tf\|_{L_B^p}\leq C_p\|f\|_{L_A^p},\quad 1<p<\infty,$
\item \label{sinintb} $\|Tf\|_{WL_B^1}\leq C_1\|f\|_{L_A^1},$
\item \label{sinintc} $\|Tf\|_{L_B^1}\leq C_2\|f\|_{H_A^1},$
\item \label{sinintd}$\|Tf\|_{{\rm{BMO}}(B)}\leq C_3\|f\|_{L^{\infty}(A)},\quad f\in L_c^{\infty}(A),$
\end{enumerate}
where $C_p,C_1,C_2,C_3>0.$
\end{lem}
\begin{proof}See J. L. Rubio de Francia \textit{et al}~\cite{jlrdffjrjlt}.
\end{proof}
Our first result is the following:
\begin{thm}\label{infvars}The variation operator $\mathcal{V}f$ maps $H^1$ to $L^1$  for $2\leq s<\infty$.
\end{thm}
\begin{proof}We have proved that the operator 
\begin{align*}
Tf(x)&=\left\{\int_{\mathbb{R}}\left(\frac{1}{2^n}\chi_{(-2^n,0)}(x-y)-\frac{1}{2^{n-1}}\chi_{(-2^{n-1},0)}(x-y)\right)f(y)\, dy\right\}_n\\
&=\int_{\mathbb{R}}K(x-y)\cdot f(y)\, dy,
\end{align*}
is a singular integral operator. Since also $\|Tf(x)\|_{l^s}=\mathcal{V}f(x)$ applying Lemma~\ref{singmaps}\,(c) to our operator $T$  shows that there exists a constant $C>0$ such that 
$$\|\mathcal{V}f\|_1\leq C\|f\|_{H^1}$$
for all $f\in H^1$.
\end{proof}
\begin{rem}It is clear that our argument also provides a different proof for the following known facts (see \cite{jkw1}):
\begin{enumerate}[(a)]
\item  $\|\mathcal{V}f\|_p\leq C_p\|f\|_p,\quad 1<p<\infty,$
\item $\|\mathcal{V}f\|_{WL^1}\leq C_1\|f\|_1,$
\item $\|\mathcal{V}f\|_{{\rm{BMO}}}\leq C_2\|f\|_{\infty},\quad f\in L_c^{\infty}(\mathbb{R}),$
\end{enumerate}
where $C_p,C_1,C_2>0.$
\end{rem} 
\begin{lem}\label{vecapfsing}Let $T$ be a singular integral operator with kernel $K\in D_r$, where $1<r<\infty$. Then, for all $1<\rho <\infty$, the weighted inequalities
$$\left\|\left(\sum_j\|Tf_j\|_B^{\rho}\right)^{1/\rho}\right\|_{L^p(w)}\leq C_{p,\rho}(w)\left\|\left(\sum_j\|f_j\|_A^{\rho}\right)^{1/{\rho}}\right\|_{L^p(w)}$$
hold if $w\in A_{p/r^\prime}$ and $r^\prime\leq p<\infty$, or if $w\in A_p^{r^\prime}$ and $1<p\leq r^\prime$. Likewise, if $w(x)^{r^\prime}\in A_1$, then the weak type inequality
\begin{align*}
w\left(\left\{x:\left(\sum_j\|Tf_j(x)\|_B^{\rho}\right)^{1/\rho}>\lambda\right\}\right)\qquad\qquad\qquad\qquad\qquad\qquad\\
\qquad\qquad \leq C_{\rho}(w)\frac{1}{\lambda}\int\left(\sum_j\|f_j(x)\|_A^{\rho}\right)^{1/\rho}w(x)\, dx
\end{align*}
holds.
\end{lem}
\begin{proof}See J. L. Rubio de Francia \textit{et al}~\cite{jlrdffjrjlt}.
\end{proof}
Our next result is the following:
\begin{thm}\label{vecapfvar}Let $2\leq s<\infty$. Then, for all $1<\rho <\infty$, the weighted inequalities
$$\left\|\left(\sum_j(\mathcal{V}f_j)^{\rho}\right)^{1/\rho}\right\|_{L^p(w)}\leq C_{p,\rho}(w)\left\|\left(\sum_j|f_j|^{\rho}\right)^{1/{\rho}}\right\|_{L^p(w)}$$
hold if $w\in A_{p/r^\prime}$ and $r^\prime\leq p<\infty$, or if $w\in A_p^{r^\prime}$ and $1<p\leq r^\prime$. Likewise, if $w(x)^{r^\prime}\in A_1$, then the weak type inequality
\begin{align*}
w\left(\left\{x:\left(\sum_j(\mathcal{V}f_j(x))^{\rho}\right)^{1/\rho}>\lambda\right\}\right)\qquad\qquad\qquad\qquad\qquad\qquad\\
\qquad\qquad \leq C_{\rho}(w)\frac{1}{\lambda}\int\left(\sum_j|f_j(x)|^{\rho}\right)^{1/\rho}w(x)\, dx
\end{align*}
holds.
\end{thm}
\begin{proof}We have proved that the operator 
\begin{align*}
Tf(x)&=\left\{\int_{\mathbb{R}}\left(\frac{1}{2^n}\chi_{(-2^n,0)}(x-y)-\frac{1}{2^{n-1}}\chi_{(-2^{n-1},0)}(x-y)\right)f(y)\, dy\right\}_n\\
&=\int_{\mathbb{R}}K(x-y)\cdot f(y)\, dy,
\end{align*}
is a singular integral operator and its kernel operator $K$ satisfies $D_r$ condition for $1\leq r<\infty$. Since also $\|Tf(x)\|_{l^s}=\mathcal{V}f(x)$ applying Lemma~\ref{vecapfsing} to our operator $T$ gives the result of our theorem.
\end{proof}
In particular we have the following corollary:
\begin{cor}Let $2\leq s<\infty$. Then the weighted inequalities
$$\left\|\mathcal{V}f\right\|_{L^p(w)}\leq C_{p,\rho}(w)\left\|f \right\|_{L^p(w)}$$
hold if $w\in A_{p/r^\prime}$ and $r^\prime\leq p<\infty$, or if $w\in A_p^{r^\prime}$ and $1<p\leq r^\prime$. Likewise, if $w(x)^{r^\prime}\in A_1$, then the weak type inequality
\begin{align*}
w\left(\left\{x:\mathcal{V}f(x)>\lambda\right\}\right)\leq C_{\rho}(w)\frac{1}{\lambda}\int |f(x)| w(x)\, dx
\end{align*}
holds.
\end{cor}

\vspace{1cm}
\noindent
Sakin Demir\\
E-mail: sakin.demir@gmail.com\\
Address:\\
İbrahim Çeçen University\\
Faculty of Education\\
04100 Ağrı, TURKEY.


\begin{thebibliography}{99}
\bibitem{jkw1}R.~Jones, R.~Kaufman' J. M.~Rosenblatt and M.~Wierdl,
\emph{Oscllation in ergodic thery}, Ergod. Th. \& Dynam. Sys. 18 (1998) 889-935.
\bibitem{jlrdffjrjlt}J. L.~Rubio de Francia, F. J.~Ruiz and J. L.~Torrea, 
\emph{Calder\'on-Zygmund theory for operator-valued kernels},
Adv. Math. 62 (1986) 7-48.

\end{thebibliography}
\end{document}